\documentclass{article}         										
\usepackage[english,strings]{babel}     
\usepackage{amsmath}           
\usepackage[utf8]{inputenc}    
\usepackage{longtable}         
\usepackage{exscale}           
\usepackage[final]{graphicx}   
\usepackage[sort]{cite}        
\usepackage{array}             
\usepackage{fancyhdr}          
\usepackage[a4paper]{geometry} 
\usepackage{xspace}            
\usepackage{tikz-cd}              
\tikzcdset{scale cd/.style={every label/.append style={scale=#1}, 
		cells={nodes={scale=#1}}}}
\usepackage{ifdraft}           
\usepackage{nchairx} 
\usepackage[sectionbib         
]{chapterbib}       
\usepackage{appendix}
\usepackage[expansion=false    
]{microtype}        
\usepackage[nottoc]{tocbibind} 
\usepackage[backref=page,      
final=true,         
pdfpagelabels       
]{hyperref}         
\usepackage{color}
\usepackage{dsfont}
\usepackage{tikz}
\usepackage{CJKutf8}
\usepackage{bbold} 
\usepackage{yfonts}
\usepackage{amssymb}
\usepackage{tikz}
\usetikzlibrary{backgrounds}

\begin{document}
\title{Tate Homology and Powered Flybys}
\author{Kevin Ruck\thanks{Institute of Mathematics, Augsburg University, Germany\newline \hspace*{1.8em}email: Kevin.Ruck@math.uni-augsburg.de}} \maketitle

\begin{abstract}
In this paper we show that in the planar circular restricted three body problem there are either infinitely many symmetric consecutive collision orbits or at least one periodic symmetric consecutive collision orbit for all energies below the first critical energy value. Using Levi-Civita regularization allows us to distinguish two different kinds of symmetric consecutive collision orbits and prove the above claim for both of them separately, one corresponding to a solar eclipse and the other to a lunar eclipse. By interpreting the orbits as Hamiltonian chords between two different Lagrangian submanifolds we can use a perturbed version of $G$-equivariant Lagrangian Rabinowitz Floer homology to prove the existence of this kind of consecutive collision orbit. To calculate this homology we show that under certain conditions the $G$-equivariant Lagrangian Rabinowitz Floer homology is equal to the Tate homology of $G$.
\end{abstract}

\section{Introduction}

Consecutive collision orbits are orbits which start and end in the centre of gravity of the planet they circle around, i.e they start and end with the spaceship crashing onto the planet. Obviously such an orbit has no real physical interpretation, it can only exist in the regularized setting. But it is known that the orbits depend continuously on the starting conditions, so for every consecutive collision orbit there is an orbit close by that barely misses the primary and flies around it very closely. Afterwards the spacecraft will be ejected from this close orbit and flies away, but eventually it will loop back around and have another almost collision. These type of orbits are still well-defined outside of the regularized setting and describe actual physical orbits used in space missions to perform for example powered flybys. The idea behind the powered flyby is that a spacecraft moving around a massive body can use its fuel more and more efficiently as it moves further and further into the potential well. This is called the Oberth effect and goes back all the way to Hermann Oberth (1923). For the application to the planing of space missions not just the pure existence of such orbits might be interesting to know, but also the orientation of these orbits. This should be especially true when it comes to planing the communication of a spacecraft with the control station on earth. If for example one wants to find an orbit that moves relatively close to the surface of a planet so that the satellite can make better observations using a consecutive collision orbit would be a possible choice, but to be able to effectively send the gathered informations back to the control station the orientation of the consecutive collision orbit should point at least roughly in the direction of the earth and not have the observed planet block the communication. 
Proving the existence of these type of orbits is one of the main results of this paper: First let us fix some notation. The massive body the space ship orbits around shell we call \textit{earth} and the second massive body is call \textit{sun}. Keeping with this picture we call a point on the straight line between the two bodies a \textit{solar eclipse point} and a point lying on the extension of this line to the other side of earth a \textit{lunar eclipse point}. \vspace{0.2cm}\\ 
\textbf{Theorem~A}\\ 
\textit{In the restricted three body problem there are for all energies below the first Lagrange point 
\begin{enumerate}
	\item[$\bullet$] infinitely many symmetric consecutive collision orbits or at least one periodic symmetric consecutive collision orbit all intersecting the symmetry axis in a solar eclipse point and
	\item[$\bullet$] infinitely many symmetric consecutive collision orbits or at least one periodic symmetric consecutive collision orbit all intersecting the symmetry axis in a lunar eclipse point.
\end{enumerate}}
Note that the symmetry we want to consider here is given by the reflection of the space coordinate at the axis between earth and sun combined with the reflection of the momentum coordinate at the axis perpendicular to the first axis. A sketch of the position space picture for two different orbits (one red, the  other blue) would look like:

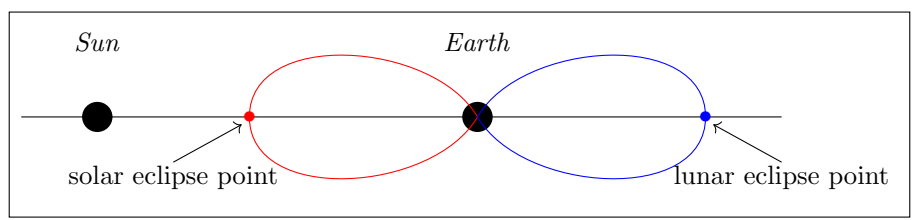
\begin{center}
\begin{tikzpicture}[framed]
\draw (1,2) --(11,2);
\path[fill] (2, 2) circle[radius=0.2];
\path[fill] (7, 2) circle[radius=0.2];
\draw[red] (7, 2) to[out=120, in=90] (4, 2);
\draw[red] (7, 2) to[out=240, in=-90] (4, 2);
\draw[blue] (7, 2) to[out=60, in=90] (10, 2);
\draw[blue] (7, 2) to[out=300, in=-90] (10, 2);
\draw[red] (4, 2) node {$\bullet$};
\draw[blue] (10, 2) node {$\bullet$};
\draw (7, 3) node {\textit{Earth}};
\draw (2, 3) node {\textit{Sun}};
\draw[arrows=->](3.0, 1.4) -- (3.9, 1.9);
\draw (3.0, 1.2) node {solar eclipse point};
\draw[arrows=->](11, 1.4) -- (10.1, 1.9);
\draw (11, 1.2) node {lunar eclipse point};
\end{tikzpicture}
\end{center}
Note that this is only a rough sketch and the actual orbits might be much more complicated then the ones drawn above. We want to remark that the existence of infinitely many consecutive collision orbits in the three body problem was already discussed in \cite{frauenfelder2019a}. But by using our approach described below we are able to significantly increase the strength of the existence statement, i.e. stating the existence of orbits, which are symmetric and have these specific orientations, is what separates this theorem from what is already known. The proof idea is to interpret this type of orbits as Hamiltonian chords starting in the fibre over the collision point and ending in the fix point set of an accordingly chosen anti-symplectic involution, i.e. one can identify those orbits with chords from one Lagrangian to another Lagrangian. After having shown that between the two Lagrangians in question there exists a Hamiltonian diffeomorphism, one then argues that in this case the chords can be interpret as critical points of the perturbed Rabinowitz action functional, where the perturbation is given by the Hamiltonian that induces the diffeomorphism between the Lagrangians. So in chapter~\ref{leaf} we give a short introduction to leaf-wise intersection points and the perturbed Rabinowitz action functional. In chapter~\ref{3BP} we discuss the planar restricted three body problem with a special emphasis on its Levi-Civita regularization. Building on the work in \cite{albers2012} there is a straight forward way of showing that the planar circular restricted three body problem is also of restricted contact type in Levi-Civita regularization (see e.g. \cite[chapter 4.2]{frauenfelder2018a}), which allows us to use Rabinowitz Floer theory to prove Theorem~A. The benefit of Levi-Civita regularization is that the Lagrangian ensuring the symmetry of our orbits becomes two different Lagrangians after the regularization. This allows us then to distinguish the different directions of the symmetric consecutive collision orbits. In chapter~\ref{ExI} we use the techniques introduced so far to prove that these specific consecutive collisions do always exist. It is important to note that at this point we are not yet able to prove any claims about the number of the orbits except from being bigger than zero. To get to the full strength of the above stated theorem we first need to introduce $G$-equivariant Lagrangian Rabinowitz Floer homology in chapter~\ref{Eq-RFH}. The definition of this homology is very straight forward, only the well-definedness of the grading needs some further arguments. After having established the $G$-equivariant version we can then prove the other main result of this paper:\vspace{0.2cm}\\ 
\textbf{Theorem~B}\\
\textit{Let $G$ be a finite group and a symmetry of the Hamiltonian system $(M,\omega,H)$ with Lagrangian $L$, which acts free. Assume that $L\cap H^{-1}(0)$ is a connected submanifold of dimension at least 1. Further, let the system be displaceable and assume that the absolute value of the Conley Zehnder index for non-constant chords is bounded below by $\text{dim}\left(L\cap H^{-1}(0)\right)$. Then the $G$-equivariant Lagrangian RF-homology is equal to the Tate homology of $G$, i.e.
\begin{align*}
RFH_{*}^{G}(M,H,L)=TH_*(G,\mathbb{Z}_2).
\end{align*}}\\
Note that displaceable in this context means that we can displace the energy hypersurface of the Hamiltonian away from the Lagrangian submanifold via a compactly supported Hamiltonian diffeomorphism.  With this new tool at hand we can proceed to prove Theorem~A in chapter~\ref{ExII}. The strategy is actually the same as in chapter~\ref{ExI}: Calculate the Rabinowitz Floer homology in an easy case and use the invariance property to transfer the result to the homology suitable to discuss symmetric consecutive collision orbits.

\subsection*{Acknowledgements}
The author would like to thank his supervisor Urs Frauenfelder for all the great input and the fruitful discussions, which made this paper possible. The author was supported by the Deutsche Forschungsgemeinschaft through DFG grant FR 2637/2-2.

\section{Leaf-wise Intersection Points}
\label{leaf}
The setting we want to work in is the following: Let $(M,\omega=\D\lambda)$ be the completion of a Liouville domain $(\widetilde{M},\lambda)$ with a Hamiltonian $H$ such that $\Sigma:=\del \widetilde{M}=H^{-1}(0)$, the support of $\D H$ is inside of a compact set $K$ and the Hamiltonian vector field coincides with the Reeb vector field on $\del\widetilde{M}$. Let $L$ be an exact Lagrangian manifold, such that $l$ with $\D l=\lambda\vert_L$ has support in $K$ and let there be an almost complex structure $J_t$, which is $\omega$-compatible and SFT like outside of $K$. The term \textit{SFT like outside of $K$} means that $J_t$ is $\omega$-compatible and fulfils
\begin{align}
J_t(x)R(x)=\del_r\big\vert_x
\end{align}
for all $x\in M\setminus K$, where $R$ is the usual extension of the Reeb vector field to the positive part of the symplectization and $\del_r$ stands for the unit vector field in the $\mathbb{R}$-direction. The name is motivated by the usual assumptions placed on $J_t$ in symplectic field theory (see for example \cite[chapter 1.4]{eliashberg2000introduction}). We call a point $p\in\Sigma$ a leaf-wise intersection point with respect to the Hamiltonian diffeomorphism $\varphi_F$ (or a Hamiltonian $F$) if there is a time $\tau$ such that $\phi_R^\tau(p)$ lies inside of $\varphi_F^{-1}(L)$, where $\phi_R^t$ is the flow of the Reeb vector field. Leaf-wise intersections were first considered by Moser in \cite{moser1978fixed} and can actually be defined in a more general setting. We chose the above set up, because it allows us to interpret the leaf-wise intersection points in terms of critical points of a perturbed Rabinowitz action functional, which was first done in \cite{albers2010leaf}. One important difference between the considerations in those papers and the above definition is that they were interested in intersections of Hamiltonian trajectories with the leaves of the Reeb vector field in both the starting and end point. In our setting we are more interested in Reeb chords starting in the Lagrangian manifold $L$ and ending in another Lagrangian, which is defined by displacing $L$ with a Hamiltonian diffeomorphism. These kind of leaf-wise intersection points were first considered by Will Merry in \cite{merry2014a}. There he also explained how to interpret them in terms of perturbed Lagrangian Rabinowitz Floer homology:
In addition to the above assumptions let $\beta:[0,1]\to\mathbb{R}$ be a smooth function with support in $\left(0,\frac{1}{2}\right)$ and
\begin{align}
\int\limits_0^1\beta(t)\D t=1,
\end{align}
and let $F$ be a smooth time dependent function such that $F(\argument,t)=0$ for all $t\in \left[0,\frac{1}{2}\right]$. Denote by $P(M,L)$ all the chords $x:[0,1]\to M$, which start and end in $L$. Then define the perturbed Rabinowitz action functional
\begin{align*}
\mathcal{A}_F^H: P(M,L)\times\mathbb{R}\to \mathbb{R}
\end{align*}
as
\begin{align}
\mathcal{A}_F^H(x,\tau):=\int\limits_0^1 x^*\lambda+ l(x(0))-l(x(1)) - \tau \int\limits_0^1\beta(t)H(x(t))\D t - \int\limits_0^1 F(x(t),t)\D t.
\end{align}
 A critical point of this functional is now a pair $(x,\tau)$ with
\begin{equation}
\begin{aligned}
\del_t x(t)=\tau\beta(t)X_H(x(t))&+X_F(x(t)),\\
\int\limits_0^1 \beta(t)H(x(t))\D t&=0.
\end{aligned}
\end{equation}
Since $\beta(t)H$ and $F$ have disjoint support the first in $\left(0,\frac{1}{2}\right)$ and the second in $\left(\frac{1}{2},1\right)$, it is not hard to show that for every leaf-wise intersection point w.r.t. $\varphi_F$ there is a corresponding critical point of the above action functional (see \cite[Propostion 2.4]{albers2010leaf} for more details). For this perturbed Rabinowitz action functional one can then define a corresponding perturbed Lagrangian Rabinowitz Floer homology $RFH(M,L,H,F)$, which is a standard procedure, so for a more detailed explanation we again refer to \cite{albers2010leaf}. By choosing a homotopy $F_s$ from $F$ to $0$, one can show as usually (\cite[Section 2.3]{albers2012infinitely}) that this perturbed homology is isomorphic to the Lagrangian Rabinowitz-Floer homology without the perturbation, i.e.
\begin{align}
RFH(M,L,H,F)\cong RFH(M,L,H).
\end{align}
This fact will be one of the most important tools on our quest to prove the existence of symmetric consecutive collision orbits in the circular restricted three body problem.

\section{The Circular Restricted three Body Problem and its Levi-Civita Regularization}
\label{3BP}
The physical setting we want to work in is the following: Given three massive bodies, which are shaped in such a way that their gravitational potential is well approximated by that of a point mass, the only relevant force is gravitation and the whole system can be described using classical mechanics. In this setting one would like to understand how these three bodies move in the superposition of their gravitational fields depending on the starting configuration.  But even though the set up we described does not seem too complicated, it turns out that it is an extremely hard problem to solve analytically. To make the problem a bit more handleable, assume that two of our bodies are much heavier than the third, such that the gravitational influence of the third one on the first two is negligible ($\to$ restricted). We further assume that the two heavy masses move around their shared centre of gravity in a circle ($\to$ circular) and that all the movements stay in one plane ($\to$ planar). This set of assumptions is known as the \textit{planar circular restricted three body problem}. Since the movement of the two big masses is already part of our assumptions, only the behaviour of the small body is left to be determined. Let's call one of the big masses $s$ for sun and the other one $e$ for earth, note that the reduced mass $\mu$ of these two bodies doesn't need to be similar to the one of the actual sun and earth. Then the Hamiltonian for this system is given by
\begin{align}
H=\frac{1}{2}\|p\|^2-\frac{\mu}{\|q-s(t)\|}-\frac{1-\mu}{\|q-e(t)\|},
\end{align}
where as usually $q\in \mathbb{R}^2$ stands for the space coordinates and $p\in\mathbb{R}^2$ stands for the momentum coordinates and the total mass is normalised to $1$. One problem with this Hamiltonian is that it is time dependent, because of the movement of earth and sun. In order to get a time independent Hamiltonian we consider a rotating coordinate system where earth and sun stay at fixed places. Hence the Hamiltonian becomes autonomous, but it also receives additional terms for the centrifugal force and the Coriolis force. So the new Hamiltonian defined on the phase space $T^*\mathbb{C}$ is 
\begin{align}
H=\frac{1}{2}\|p\|^2+p_1q_2-p_2q_1-\frac{\mu}{\|q-s\|}-\frac{1-\mu}{\|q-e\|}-E,
\end{align}
where $e$ and $s$ are now fixed coordinates and $E$ is the energy of our system we want to consider. Even though the Hamiltonian is now autonomous, there is still one important issue we have to deal with namely the collisions. The standard approach is to choose the starting conditions of our small body in such a way that he can reach only one of the other bodies - let's choose the earth $e$ for now - and then use the Moser regularization around this body. We follow this approach, but instead of Moser we will use Levi-Civita regularisation, since this will help us later to distinguish different kinds of orbits. 

The starting point of the Levi-Civita regularization is the simple map
\begin{align*}
l: \mathbb{C}\setminus \{0\} \to \mathbb{C}\setminus \{0\}\ \ ;\ z\mapsto z^2
\end{align*}
and its cotangent lift
\begin{align*}
\mathfrak{L}: T^*(\mathbb{C}\setminus\{0\}) \to T^*(\mathbb{C}\setminus\{0\})\ \ ;\ (z,w)\mapsto \left(l(z), \frac{w}{\overline{l'(z)}}\right).
\end{align*}
Note that we shift our coordinate system, such that $e$ lies now in the origin. To regularize a Hamiltonian system $(T^*\mathbb{C}, H, \omega=\D\lambda)$ we first pull it back using the map $\mathfrak{L}$ and get the new system $(T^*\mathbb{C}, \widehat{H}:=\mathfrak{L}^* H, \widehat{\omega}=\mathfrak{L}^*\omega=\D \mathfrak{L}^*\lambda)$. 
Note that $X_{\mathfrak{L}^*H}=\mathfrak{L}^*X_H$ because
\begin{align*}
\mathfrak{L}^*i_{X_H}\omega &=\mathfrak{L}^*\D H \\
 \Leftrightarrow i_{\mathfrak{L}^*X_H}\widehat{\omega}&=\D\widehat{H}
\end{align*}
So if we have a curve $\gamma$ satisfying $\frac{\D}{\D t}\gamma(t)=X_{\widehat{H}}(\gamma(t))$ we can recover the Hamiltonian dynamics on the original space by applying the map $\mathfrak{L}$:
\begin{align*}
\frac{\D}{\D t}\mathfrak{L}(\gamma(t))=\D \mathfrak{L}\left(\frac{\D}{\D t}\gamma(t)\right)=\D \mathfrak{L} \left(X_{\widehat{H}}(\gamma(t))\right)= X_H(\mathfrak{L}(\gamma(t)))
\end{align*}
The second step is now to multiply the pulled back Hamiltonian with $\|z\|^2$. On the energy hypersurface this changes the Hamiltonian vector field only by a prefactor, since
\begin{align*}
i_{X_{\|z\|^2\widehat{H}}}\omega=\D (\|z\|^2\mathfrak{L}^*H) =\|z\|^2\D \left(\mathfrak{L}^*H\right) + \widehat{H}\D \|z\|^2=\|z\|^2\D \left(\mathfrak{L}^*H\right)
\end{align*}
on $(\mathfrak{L}^*H)^{-1}(0)$. Therefore we can finally conclude that the trajectories of $(H,\omega)$ and $(\|z\|^2\widehat{H},\widehat{\omega})$ lying in their respective energy hypersurfaces coincide up to parametrization. To be able to apply the regularization procedure to our situation, we need to put the earth in the centre of our coordinate system, i.e. shift $q \to q+e$. Note that we can choose our coordinate system in such a way that the two big masses lie on the real axis. The regularised Hamiltonian is now given by
\begin{equation}
\begin{aligned}
&K(z,w)= \|z\|^2\cdot(\mathfrak{L}^*H)(z,w)\\
&=\frac{1}{8}\|w\|^2-E\|z\|^2-(1-\mu)+\frac{\|z\|^2-e_1}{2}(w_1z_2-w_2z_1)-\frac{\mu\|z\|^2}{\|z^2+e-s\|}.
\end{aligned}
\label{3Ham}
\end{equation}
Since we want to use Rabinowitz Floer homology, we first need to show that we are in fact in a setting, in which the Rabinowitz Floer homology is well-defined. Gladly most of this work was already done by \cite{albers2012}. Here it was shown that below the first critical energy values the energy hypersurface around the earth and the sun, respectively, are of restricted contact type. This means we can consider the hypersurface to be the boundary of a Liouville domain and the whole space to be the completion of this Liouville domain. But in the above mentioned paper they use the Moser regularization instead of the Levi-Civita regularization, so we first need to transfer this statement to our case. The Idea is to show that we can pull back every closed fiberwise star-shaped hypersurface of $T^*S^2$ to a star-shaped hypersurface in $\mathbb{C}^2$ (see \cite[chapter 4.2]{frauenfelder2018a}). For this one considers the local diffeomorphism
\begin{align*}
\widehat{\mathfrak{L}}: T^*\left(\mathbb{C}\setminus \{0\}\right) \to T^*S^2\setminus\left(S^2 \cup T^*_NS^2\right),
\end{align*}
which is a composition of the map $\mathfrak{L}$ from above, the switch $\mathtt{sw}$ of fibre coordinates with base point coordinates $(q,p)\mapsto (-p,q)$ and the inverse stereographic projection $P_N^{-1}$ at the north pole. This maps the energy hypersurface $K^{-1}(0)$ in Levi-Civita regularization to the energy hypersurface in Moser regularization. Now we pull back the symplectic two form and the Liouville one form. 
For the symplectic two form we have on $T^*\left(\mathbb{C}\setminus \{0\}\right)$
\begin{align*}
\widehat{\mathfrak{L}}^* \omega_{\text{Moser}}&=(P_N^{-1}\circ\mathtt{sw}\circ \mathfrak{L})^* \omega_{\text{Moser}}\\&=\mathfrak{L}^*\RE(\D q\wedge \overline{\D p})\\ &=\RE\left(2z\D z \wedge\overline{\left(\frac{1}{2\bar{z}}\D w-\frac{w}{2\bar{z}^2}\D \bar{z}\right)}\right)\\ &=\RE(\D z\wedge\overline{\D w})\\ &= \D z_1\wedge\D w_1+\D z_2\wedge\D w_2,
\end{align*}
with the not so standard notation of $z=z_1+iz_2$ and $z_1,z_2\in \mathbb{R}$. For the Liouville one form we have: 
\begin{align*} 
\widehat{\mathfrak{L}}^* \lambda_{\text{Moser}}&=(P_N^{-1}\circ\mathtt{sw}\circ \mathfrak{L})^* \lambda_{\text{Moser}}\\&=\mathfrak{L}^*  \RE(q\overline{\D p})\\ 
&= \RE\left(z^2\overline{\D\left(\frac{w}{2\bar{z}}\right)}\right)\\ 
&= \frac{1}{2}\RE(z\overline{\D w}-\overline{w}\D z)\\ 
&= \frac{1}{2} \left(z_1\D w_1- w_1\D z_1+ z_2\D w_2- w_2\D z_2\right) 
\end{align*}

With this we see that the Liouville vector field implicitly defined by $ i_X\omega = \lambda$ is therefore $\frac{1}{2}\left(z\frac{\del}{\del z}+ w\frac{\del}{\del w}\right)$, i.e. it is a radial vector field. For Moser regularization we already know from \cite{albers2012} that $X_M H_M>0$ for all energies below the first critical energy value, where $X_M$ should stand for the Liouville vector field in Moser regularisation and $H_M$ for the corresponding Hamiltonian. Since we can get both the Liouville one form and the Hamiltonian function in Levi-Civita regularization by pulling them back from Moser regularization we see that
\begin{align}
(\widehat{\mathfrak{L}}^*X_M)(\widehat{\mathfrak{L}}^*H_M)=\widehat{\mathfrak{L}}^*(X_M H_M)>0.
\end{align}
Note that since both hypersurfaces are compact and the involved vector fields and functions are smooth this result extents to all of $K^{-1}(0)$. So we can finally conclude that also for Levi-Civita regularization we can view the energy hypersurface as boundary of a Liouville domain and the whole space as the completion of the Liouville domain. To finally meet all the requirements for a well-defined RF-homology we need to cut our Hamiltonian off outside of the Liouville domain and set it to a constant $c_0$ which is different from the energy we want to consider. Further we need to show that on the energy hypersurface the Hamiltonian vector field coincides with the Reeb vector field. But since we already know that our energy hypersurface is the boundary of the Louville domain, the Reeb vector field is just $R=f\cdot X_K$ for a smooth function $f$, which is never zero. Therefore we can define a new Hamiltonian $\widetilde{H}:= f\cdot\widehat{H}
$, which results in a new Hamiltonian vector field $X_{\widetilde{H}}=f\cdot X_{\widehat{H}}$. From this we can conclude two things: First the trajectories of the new Hamiltonian are just a reparametrization of the old ones. Second the Hamiltonian vector field $X_{\widetilde{H}}$ coincides with the Reeb vector field, because $X_{\widetilde{K}}=f\cdot X_K=R$ and $\widetilde{K}^{-1}(0)= (f\cdot K)^{-1}(0)= K^{-1}(0)$ (remember that $f$ is never zero). So finally we have checked all the necessary requirements to define the Rabinowitz Floer homology. For simplicity we will denote the Hamiltonian after Levi-Civita regularization and with the above adjustment by a suitable function $f$ again by $H$.  

\section{Existence of Symmetric Consecutive Collision Orbits I}
\label{ExI}
Now we want to use what we know so far to prove that in the planar circular restricted three body problem in Levi-Civita regularisation (around the earth) for energies below the first critical energy value there are always at least two symmetric consecutive collision orbits, for one the intersection point of the orbit with its symmetry axis lies on the line between the sun and earth, for the other one  exactly on the opposite side. 
To prove the above claim we will first interpret the symmetric consecutive collision orbits as Reeb chords between two different Lagrangian submanifolds in the following way: On the original space of the planar restricted three body problem we have an anti-symplectic involution
\begin{align}
R: T^*\mathbb{C} \to T^*\mathbb{C}\ \ ;\ q\mapsto \overline{q}\ ,\ p\mapsto -\overline{p},
\end{align}
which leaves the Hamiltonian invariant. After the Levi-Civita regularization this anti-symplectic involution corresponds to now two involutions
\begin{align}
\widehat{R}_1: z\mapsto -\overline{z}\ \ ;\ w\mapsto \overline{w}\\
\widehat{R}_2: z\mapsto \overline{z}\ \ ;\ w\mapsto -\overline{w}
\end{align}
and their corresponding fix point sets are
\begin{align}
L_{S}:=\text{Fix}(\widehat{R}_1)=i\mathbb{R}\times\mathbb{R}\ \ \text{and}\ \ 
L_{M}:=\text{Fix}(\widehat{R}_2)=\mathbb{R}\times i\mathbb{R}.
\end{align}
If we think of the small body as the moon and assume that we chose the coordinates such that the sun is at the left side of the earth, $L_S$ would correspond to the eclipse of the sun (since $l:i\mathbb{R}\to \mathbb{R}^-_0$) and $L_M$ would correspond to the eclipse of the moon (since $l:\mathbb{R}\to \mathbb{R}^+_0$), at least as long as we stay below the first critical energy value. Since the symplectic form is given by $\omega= \D z_i\wedge\D w_i$ these two set are obviously Lagrangian subspaces. The reason why we consider this anti-symplectic involution is the following:
Let $x(t)$ be a solution of $H$ with $x(0)\in L_{\text{col}}$ and $x(1)\in L_{S}$ (or equivalently $x(1)\in L_{M}$). Here $L_{\text{col}}$ is the Lagrangian subspace of collision, i.e. $L_{\text{col}}:=\left\{(q,p)\in\mathbb{R}^4\ \vert\ q=0\right\}$. Then $\widehat{R}_1(x(1-t))$ is also a solution of $H$ since
\begin{equation}
\begin{aligned}
\frac{\D}{\D t}\big\vert_{t_0}\widehat{R}_1(x(1-t))&=\D\widehat{R}_1 \frac{\D}{\D t}\big\vert_{t_0}x(1-t)\\&=-\D\widehat{R}_1 \frac{\D}{\D t}\big\vert_{1-t_0}x(t)\\&=-\D\widehat{R}_1 X_H(x(t))\big\vert_{1-t_0}\\&=X_H\widehat{R}_1(x(t)))\big\vert_{1-t_0}\\&=X_H\widehat{R}_1(x(1-t)))\big\vert_{t_0}.
\end{aligned}
\end{equation}
Further, the requirement $x(1)\in L_{S}$ guaranties that the concatenation $x\circ\widehat{R}_1(x(1-\argument))$ is still a smooth solution of $H$. This means, if we find a chord from $L_{S}$ to $L_{\text{col}}$, we also automatically find a symmetric consecutive collision orbit.

Our goal is now to view the chords from $L_S$ or $L_M$ to $L_{\text{col}}$ as leaf-wise intersection points (see Section~\ref{leaf}). To do this we need Hamiltonians $F_S$ or $F_M$ such that 
\begin{align}
\varphi_{F_S}^{-1}(L_S)= L_{\text{col}}\ \ \text{ and }\ \ \varphi_{F_M}^{-1}(L_M)= L_{\text{col}},
\end{align}
where $\varphi_{F_S}$ or $\varphi_{F_M}$ is the Hamiltonian flow. So define
\begin{align}
F_S(z,w)=\pi(z_2^2+w_2^2)\ \ \text{ and }\ \ F_M(z,w)=\pi(z_1^2+w_1^2).
\end{align}
The Hamiltonian vector field of $F_S$ is then 
\begin{align}
X_{F_S}= \frac{\pi}{2}\cdot w_2\frac{\del}{\del z_2}-\frac{\pi}{2}\cdot z_2\frac{\del}{\del w_2}
\end{align}
and the corresponding flow is given by
\begin{align}
\varphi_{F_S}(z,w,t)=\begin{pmatrix} z_1 \\ \cos\left(\frac{\pi}{2} t\right)z_2+\sin\left(\frac{\pi}{2} t\right)w_2 \\ w_1 \\ \cos\left(\frac{\pi}{2} t\right)w_2-\sin\left(\frac{\pi}{2} t\right)z_2 \end{pmatrix}.
\end{align}
With this it is not hard to see that $\varphi_{F_S}\left(L_S,t=1\right)=L_{\text{col}}$, for $L_M$ the calculations are almost the same. Note that both $F_S$ and $F_M$ are invariant under the $\mathbb{Z}_2$ action $z\mapsto -z$ and $w\mapsto -w$. To be able to interpret our chords as leaf-wise intersection points, we will have to turn $F_S$ or $F_M$ respectively into a compactly supported function. For this choose a compact set $\widetilde{K}$ such that $H^{-1}\subseteq\widetilde{K}$ and a corresponding bump function $\chi_{\widetilde{K}}$, which is $1$ on $\widetilde{K}$ and compactly supported. Define 
\begin{align*}
\widehat{F}_S(z,w)= \chi_{\widetilde{K}}(z^2,w^2)\cdot F_S(z,w)\ \ \text{ and }\ \ \widehat{F}_M(z,w)= \chi_{\widetilde{K}}(z^2,w^2)\cdot F_M(z,w),
\end{align*}
so that we don't loose the invariance with respect to the $\mathbb{Z}_2$ action of our Hamiltonian functions. By again choosing a suitable bump function $\zeta:[0,1]\to \mathbb{R}$ we finally get the desired form of the Hamiltonian functions, i.e.
\begin{align*}
\widehat{F}_S^t(z,w):=\zeta(t)\widehat{F}_S(z,w)\ \ \text{ and }\ \ \widehat{F}_M^t(z,w):=\zeta(t)\widehat{F}_M(z,w)
\end{align*}
with $\widehat{F}_S^t(z,w)=0=\widehat{F}_M^t(z,w)$ for all $t\in\left[0,\frac{1}{2}\right]$ and $\zeta(t)=1$ in a neighbourhood around $1$.

As a next step we want to show that the critical points of $\mathcal{A}^{\widehat{F}_S^t}_H$ are in one to one correspondence with the solutions of the Hamiltonian equation of $H$ with energy $H=0$ that start in $L_{\text{col}}$ and end in $L_S$. So first let $(y,\eta)$ be a critical point of $\mathcal{A}^{\widehat{F}_S^t}_H$ with $y(0),y(1)\in L_{\text{col}}$. Then as discussed above in $[0,\frac{1}{2}]$ we have
\begin{align}
\del_t y= \eta \beta(t)X_H(y(t)).
\end{align}
By Picard-Lindelöf there is also a unique solution of 
\begin{align}
\begin{cases}
\del_t y=\eta X_H(x(t))\\
x(0)=y(0) 
\end{cases}
\end{align}
in $[0,1]$. Then consider $\widetilde{x}(t):=x\left( \int\limits_0^t\beta(T)\D T\right)$ and see that 
\begin{align}
\del_t \widetilde{x}=\eta\beta(t)X_H(\widetilde{x})
\end{align}
in $[0,\frac{1}{2}]$. By uniqueness this means that $y(t)=\widetilde{x}(t)$ and therefore
\begin{align}
x(1)=x\left( \int\limits_0^{\frac{1}{2}}\beta(T)\D T\right)=y\left(\frac{1}{2}\right) \in \varphi^{-1}_{\widehat{F}_S^t}(L_{\text{col}})\cap H^{-1}(0)=L_S\cap H^{-1}(0)
\end{align}
On the other hand let $x\left(\frac{t}{\tau}\right)$ be a solution of the Hamiltonian equation of $H$ with $x(0)\in L_{\text{col}}$, $x(1)\in L_S$. Then take the usual cut off function $\beta(t)$ and define 
\begin{align}
y(t):= \begin{cases}
x\left( \int\limits_0^t\beta(T)\D T\right)& t\in[0,\frac{1}{2}]\\
\varphi^{t}_{\widehat{F}_S^t}(x(1)) &  t\in[\frac{1}{2},1]
\end{cases},
\end{align}
which is clearly a critical point of $\mathcal{A}^{\widehat{F}_S^t}_H$. This means for every leaf-wise intersection point there exists a corresponding symmetric consecutive collision orbit, so we now only have to show the existence of a leaf-wise intersection point. 

The next step is to calculate the homology $RFH(M,L_{\text{col}},H,\widehat{F}^t_S)$. First we know that 
\begin{align}
RFH(M,L_{\text{col}},H,\widehat{F}^t_S)\cong RFH(M,L_{\text{col}},H).
\end{align}
In $\mathbb{C}^4$ we can always displace a subspace from a compact set and since by assumption our energy hypersurface is compact, \cite[chapter 2.4]{merry2014a} tells us that the homology $RFH(M,L_{\text{col}},H)$ is zero and so we can conclude:
\begin{align}
RFH(M,L_{\text{col}},H,\widehat{F}_S)=0
\end{align}
Both $L_S\cap L_{\text{col}}$ and $L_M\cap L_{\text{col}}$ are one dimensional subspaces and since $H^{-1}(0)$ is a star-shaped hypersuface, $H^{-1}(0)$ has to intersect with $L_S\cap L_{\text{col}}$ and $L_M\cap L_{\text{col}}$ in two points each. These points are automatically fix points of $\widehat{F}_S$ and $\widehat{F}_M$, respectively. If we set for those points the Lagrange multiplier $\tau$ to zero, we get critical points of the perturbed Rabinowitz action functional. But since we know that the homology is zero and that two constant points can't cancel each other out in the homology, there needs to be at least one nonconstant critical point. This then finally tells us that there exists at least one symmetric consecutive collision orbit corresponding to $L_S$ and one to $L_M$ for all energies below the first critical energy value.

Even though knowing of the existence of at least one symmetric consecutive collision orbit is already quite useful, we can still improve it by using an equivariant version of the Lagrangian Rabinowitz Floer homology. The improvement we want to archive in the next sections is the existence of not just one symmetric consecutive collision orbit but of infinitely many or at least of one periodic symmetric consecutive collision orbit. 

\section{Equivariant Lagrangian Rabinowitz Floer Homology}
\label{Eq-RFH}
In this section we want to understand what equivariant Rabinowitz Floer-homology is and when we can use it to gain better inside of the Hamiltonian system we are interested in.

Let $(M,\omega=\D\lambda)$ be the completion of a Liouville domain $(\widetilde{M},\lambda)$ with a Hamiltonian $H$ such that $\Sigma:=\del \widetilde{M}=H^{-1}(0)$, the support of $\D H$ is inside of a compact set $K$ and the Hamiltonian vector field coincides with the Reeb vector field on $\del\widetilde{M}$. Let $L$ be an exact Lagrangian submanifold, such that $l$ with $\D l=\lambda\vert_L$ has support in $K$ and let there be an almost complex structure $J_t$, which is $\omega$-compatible and SFT like outside of $K$. To be able to use a $\mathbb{Z}$ grading further assume that the morphism induced by the Maslov index of disks with Lagrangian boundary
\begin{align}
\mu: \pi_2(M,L)\to \mathbb{Z}
\end{align}
is trivial. Now let $G$ be a finite Lie group that acts free and proper on the symplectization $\Sigma\times \mathbb{R}$, is a symmetry of our Hamiltonian system, leaves $L$ invariant and $J_t$ is equivariant w.r.t $G$ in the following sense:
\begin{align}
\D\phi_g J_t(x)=J_t(g\acts x)\D\phi_g\ \ \ \text{ for all }g\in G
\end{align}
Note that $\phi_g$ is defined by $\phi_g(x)=g\acts x$. Then we define the action of $G$ on $P_c(M,L)\times\mathbb{R}$ by acting point-wise on the first component and trivial on the second component. Note that $P_c(M,L)$ stands for the chords in $M$, which are relative contractible with respect to $L$, and that a well-defined $\mathbb{Z}$ grading can only be achieved for this subset of chords. The $G$-equivariant Lagrangian RF chain complex $CRF_*^G(M,L)_a^b$ is then constructed by taking (in the Morse-Bott situation) a $G$-invariant Morse function on a compact set of critical points $crit(\mathcal{A}_H)_a^b$, form the usual $\mathbb{Z}_2$ vector space with them and then take the quotient with respect to the $G$ action. As usually the restriction to critical points $x$ with $a\le \mathcal{A}(x)\le b$ is necessary for compactness reasons. To still have a well-defined grading by the CZ-index on the quotient space, we need to show the following proposition:
\begin{proposition}
Let $x$ and $y$ be two Hamiltonian chords in the above described setting with $g\acts x=y$ for a $g\in G$. Then $x$ and $y$ have the same CZ-index.
\end{proposition}
\begin{proof}
We know that the following connection between the CZ-index and the Fredholm index 
\begin{align*}
\mu_{\text{CZ}}(x)-\mu_{\text{CZ}}(x^\prime)=\text{ind}(\text{D}s(u))
\end{align*}
holds, where $u\in\mathcal{H}_{x,x^\prime}$ is a path in the chord space with $x$ and $x^\prime$ as asymptotics and $s$ is as usually the section given by the Floer equation, i.e.
\begin{align*}
s(u)= \del_s u+J_t(u)(\del_t u-\tau X_H(u)).
\end{align*}
Note that if $u$ is not a solution of the Floer equation we first need to choose a connection on the $L^2$ vector bundle over $\mathcal{H}_{x,x^\prime}$ to be able to define the vertical differential. Since our assumptions on the action of $G$ are chosen in such a way that 
\begin{align*}
\D\phi_g s(u)=s(g\acts u)\ \ \ \text{ for all }g\in G.
\end{align*}
$\D\phi_g$ induces then an isomorphism between the kernels of $\text{D}s(u)$ and $\text{D}s(g\acts u)$ and also between their cokernels, hence the Fredholm index of those two is the same. With this we get
\begin{align*}
\mu_{\text{CZ}}(x)-\mu_{\text{CZ}}(x^\prime)=\text{ind}(\text{D}s(u))=\text{ind}(\text{D}s(g\acts u))=\mu_{\text{CZ}}(g\acts x)-\mu_{\text{CZ}}(g\acts x^\prime),
\end{align*}
which is equivalent to
\begin{align*}
\mu_{\text{CZ}}(g\acts x)-\mu_{\text{CZ}}(x)=\mu_{\text{CZ}}(g\acts x^\prime)-\mu_{\text{CZ}}(x^\prime).
\end{align*}
Then define the map
\begin{align*}
h:G\to\mathbb{Z};\ \ \ g\mapsto \mu_{\text{CZ}}(g\acts x)-\mu_{\text{CZ}}(x),
\end{align*}
note that the above calculation shows that $h$ is independent of $x$. Now we show that $h$ is a morphism between $G$ and $\mathbb{Z}$:
\begin{align*}
h(g_1\cdot g_2)&=\mu_{\text{CZ}}(g_1\cdot g_2\acts x)-\mu_{\text{CZ}}(x)\\
&=\mu_{\text{CZ}}(g_1\acts( g_2\acts x))-\mu_{\text{CZ}}(g_2\acts x)+\mu_{\text{CZ}}(g_2\acts x)-\mu_{\text{CZ}}(x)\\
&=h(g_1)+h(g_2)
\end{align*}
\underline{Claim:}
A morphism between a finite group and $\mathbb{Z}$ is trivial.\\
\underline{Proof of the Claim:}
Assume there is a $g\in G$ with $h(g)\ne 0$. Define $N:=\max\limits_{x\in G}\lvert h(x)\rvert$, then there is a $M\in\mathbb{Z}$ such that $M\cdot h(g)>N$. But this means that
\begin{align*}
N<M\cdot h(g)<\lvert h(g^{\lvert M\rvert})\rvert\le N,
\end{align*}
 which is a clear contradiction. This proves the Proposition.
\end{proof}
The differential of this complex is defined by counting $G$-equivalence classes of unparametrized gradient flow lines:
\begin{equation}
\begin{aligned}
\del&:CRF_*^G(M,L)_a^b\to CRF_*^G(M,L)_a^b\\
\del [x]&:= \sum\limits_{\substack{[y]\text{ with}\\\mu_{\text{CZ}}([y])=\mu_{\text{CZ}}([x])+1}}\#_2\left\{[u]\ \bigg\vert\ \widetilde{x}\overset{u}{\longrightarrow}\widetilde{y} \text{ gradient flow, }\widetilde{x}\in[x]\text{ and }\widetilde{y}\in[y]\right\}
\label{Gdiff}
\end{aligned}
\end{equation}
The group action on the moduli space is of course again defined point-wise. The final question we have to answer about the $G$-equivariant RF-homology is, if the differential is still a well-defined differential.
\begin{proposition}
Under the current assumptions the map defined in equation~(\ref{Gdiff}) is well-defined and fulfills $\del^2=0$.
\end{proposition}
\begin{proof}
First note that $G$ acts free on the moduli space since if there is a $g\in G$ such that $(g\acts u)=u$, then $g\acts u(s,t)=u(s,t)$ for all $s,t\in[0,1]$ and because $G$ acts free on $\Sigma\times\mathbb{R}$, this implies $g=1_G$. From the standard construction of the RF-homology we know that the set of unparametrized gradient flow lines define a manifold $\mathcal{M}(x,y)$ with dimension $\mu(x)=\mu(y)-1$. The corresponding $G$-invariant set is then 
\begin{align}
\mathcal{M}_G([x],[y]):=\left(\bigcup\limits_{\widetilde{x}\in[x],\ \widetilde{y}\in[y]} \mathcal{M}(\widetilde{x}, \widetilde{y})\right)/G
\end{align}
Since $G$ is only a finite group, $\bigcup\limits_{\widetilde{x}\in[x],\ \widetilde{y}\in[y]} \mathcal{M}(\widetilde{x}, \widetilde{y})$ is a finite disjoint union of manifolds and hence again a manifold of same dimension. Further $G$ acts also proper, since it is finite. Now we know that a manifold quotient by a free and proper action is again a manifold with $\text{dim}\mathcal{M}(\widetilde{x},\widetilde{y})-\text{dim} G$ and  by assumption $G$ has dimension zero.
From this we see that $\text{dim}\mathcal{M}(\widetilde{x},\widetilde{y})=\mu([x])-\mu([y])-1$ and the usual argument for $\del^2=0$ and the well-definiteness of the map also works in the $G$-equivariant setting.
\end{proof}

It is not hard to see that one can also define the perturbed Rabinowitz Floer homology from Section~\ref{leaf} in the equivariant setting by simply choosing an invariant perturbation $F$. The arguments, which show that the homology is independent of the chosen manifold and Hamiltonian up to homotopy, still work the same as in the normal case since the $G$-invariant smooth functions are also path connected. 

There is now a very interesting connection between the $G$-equivariant RFH and the Tate homology of $G$. But before we can investigate this further, we first need to repeat some important things connected to Tate homology. 
\begin{definition}\footnote{See \cite{brown1982a}[§VI.3], note that we use $\mathbb{Z}_2$ instead of $\mathbb{Z}$.}
A complete resolution (for a finite group $G$) is an acyclic complex $F = (F_i)_{i\in\mathbb{Z}}$ of projective $\mathbb{Z}_2[G]$-modules,
together with a map $\epsilon: F_0 \to \mathbb{Z}_2$ such that $F_+ \stackrel{\epsilon}{\longrightarrow} \mathbb{Z}_2\longrightarrow 0$ is a resolution in the usual sense, where $F_+ = (F_i)_{i\ge 0}$.
\end{definition} 
Remember that given a ring $R$ and an $R$-module $M$ a resolution of $M$ is an exact sequence of $R$-modules 
\begin{align}
\cdots \longrightarrow F_2\stackrel{\del_2}{\longrightarrow}F_1\stackrel{\del_1}{\longrightarrow}F_0\stackrel{\epsilon}{\longrightarrow}M\longrightarrow 0.
\end{align}
The following property of complete resolutions is crucial for the well-definedness of Tate homology:
\begin{proposition}
If $\epsilon: F \to \mathbb{Z}_2$ and $\epsilon^\prime: F^\prime \to \mathbb{Z}_2$ are complete resolutions, then there
exists a unique homotopy class of augmentation-preserving maps from $F$ to $F^\prime$.These maps are homotopy equivalences.
\label{wellTate}
\end{proposition}
See \cite{brown1982a}[Proposition~3.3] for the proof. 
\begin{definition}\footnote{We took the definition from \cite{brown1982a}[§VI.4], but it goes back to \cite{tate1952a}.}
Let $F = (F_i)_{i\in\mathbb{Z}}$ be a complete resolution for the finite group $G$. The Tate homology of $G$ with coefficients in a $G$-module $M$ is defined by
\begin{align}
TH_*(G,M):= H_*(F\tensor_G M).
\end{align}
\end{definition}
Note that proposition~\ref{wellTate} tells us that the Tate homology is unique up to isomorphism and therefore well-defined. This also means that if we want to compute the Tate homology of a group $G$ we can choose any acyclic complex to do so, as long as the complex is a complete resolution. This fact is the main idea behind the connection between $G$-equivariant RF-homology and Tate homology:
\begin{theorem}
\label{tate}
Let $G$ be a finite group and a symmetry of the Hamiltonian system $(M,\omega,H)$ with Lagrangian $L$, which acts free. Assume that $L\cap H^{-1}(0)$ is a connected submanifold of dimension at least 1 and that the system $(M,\omega,H,L)$ fulfils all the requirements needed for Lagrangian Rabinowitz Floer homology. Further, let the system be displaceable and the Conley Zehnder index $\mu_{CZ}$ for the non-constant chords of $(M,\omega,H)$ fulfil
\begin{align}
\lvert\mu_{CZ}(x)\rvert>\text{dim} \left(L\cap H^{-1}(0)\right).
\label{con1}
\end{align}
Then the $G$-equivariant Lagrangian RF-homology is equal to the Tate homology (with $\mathbb{Z}_2$ coefficients) of $G$, i.e.
\begin{align}
RFH_{*}^{G}(M,H,L)=TH_*(G,\mathbb{Z}_2).
\end{align}
\end{theorem}
\begin{remark}
The statement is also true for equivariant Rabinowitz Floer homology for loops instead of chords between a Lagrangian, if one changes the condition (\ref{con1}) to $\lvert\mu_{CZ}\rvert>\text{dim}H^{-1}(0)$. The proof then also works exactly the same. The reason why we chose to state the theorem like this is, because it fits better for the application to consecutive collision orbits we are discussing in this paper.
\end{remark}
\begin{proof}[Proof of Theorem~\ref{tate}]
The idea of this proof is to show that the Lagrangian Rabinowitz Floer complex for displaceable systems is a complete resolution for $G$. The displaceability guaranties that our Floer complex is acyclic and by definition consists of $\mathbb{Z}_2$ vector spaces. Since $G$ acts by assumption on the generators of the complex, we can view each space in the complex as free $\mathbb{Z}_2[G]$ module. We know that $L\cap H^{-1}(0)$ is at least one dimensional and by condition (\ref{con1}) this means that 
\begin{align*}
CRF_1(M,H,L)\overset{\del_1}{\longrightarrow}CRF_0(M,H,L)
\end{align*}
is equal to the corresponding part of the Morse complex 
\begin{align*}
CM_1(L\cap H^{-1}(0),f)\overset{\del_1}{\longrightarrow}CM_0(L\cap H^{-1}(0),f),
\end{align*}
where $f$ is a $G$-invariant Morse function on $L\cap H^{-1}(0)$. Note that here we assume the grading to be $\mu=\mu_{CZ} +\mu_{Morse}$. If one would like to use the signature index instead of the Morse index, one would need to include an index shift in the statement of the theorem.  Our goal is now to find an augmentation $\epsilon: CRF_0(M,H,L)=CM_0(L\cap H^{-1}(0),f)\to \mathbb{Z}_2$, such that $\left(CRF_i(M,H,L)\right)_{i\ge 0} \stackrel{\epsilon}{\longrightarrow} \mathbb{Z}_2\longrightarrow 0$ is a resolution of $\mathbb{Z}_2$. It is well known that for every connected manifold the zeroth homology is one dimensional as long as it is not empty. This means that the image of $\del_1$ is a codimension one subspace in side $CM_0(L\cap H^{-1}(0),f)$ and therefore we can define $\epsilon$ as the linear map that maps all elements in the image of $\del_1$ to zero and the basis vector in the complement to one. With this definition it is clear that $\text{im}(\del_1)=\ker(\epsilon)$ and that $\epsilon$ is surjective, i.e. 
\begin{align*}
\ldots \overset{\del_2}{\longrightarrow}CRF_1(M,H,L)\overset{\del_1}{\longrightarrow}CRF_0(M,H,L)\overset{\epsilon}{\longrightarrow}\mathbb{Z}_2\longrightarrow 0
\end{align*}
is a resolution of $\mathbb{Z}_2$ and $\left(CRF_i(M,H,L)\right)_{i\in\mathbb{Z}}$ is a complete resolution. The corresponding Tate homology is then $TH_*(G,\mathbb{Z}_2):= H_*(CRF_*(M,H,L)\tensor_G \mathbb{Z}_2)$, where we consider $\mathbb{Z}_2$ as a trivial $G$-module. Tensoring over $G$ to the $\mathbb{Z}_2[G]$ modules $CRF_i(M,H,L)$ the trivial $G$-module $\mathbb{Z}_2$ corresponds to taking the quotient of the $CRF_i(M,H,L)$'s with respect to the $G$ action. The new boundary operators between the modules $CRF_i(M,H,L)\tensor_G \mathbb{Z}_2$ are just the old boundary operators applied to the first part of the tensor product, i.e. $\widetilde{\del}(x\tensor_G 1):= (\del x)\tensor_G 1$. A quick calculation shows that these new boundary operators are exactly the same as the ones we defined for the $G$-equivariant RF-homology. Hence the two complexes $(CRF_*(M,H,L)\tensor_G \mathbb{Z}_2,\widetilde{\del})$ and $(CRF^G_*(M,H,L),\del_G)$ are the same and therefore also their respective homologies. This concludes the proof of the theorem.
\end{proof}

\section{Existence of Symmetric Consecutive Collision Orbits II}
\label{ExII}
After the preparations in the previous section we are now able to prove the existence statement for symmetric consecutive collision orbits in its full strength:
\begin{theorem}
In the setting of the planar circular restricted three body problem there are
\begin{enumerate}
	\item[$\bullet$] infinitely many symmetric consecutive collision orbits or at least one periodic symmetric consecutive collision orbit all intersecting their symmetry axis on the straight line between the second and the main body and
	\item[$\bullet$] infinitely many symmetric consecutive collision orbits or at least one periodic symmetric consecutive collision orbit all intersecting their symmetry axis on the extension of this line to the opposite side of the main body
\end{enumerate}
for all energies below the first critical energy value. 
\label{maxEx}
\end{theorem}
\begin{remark}
If we would consider as main body the earth, as second body the sun and as small satellite the moon, then for the first type of orbits the point where it intersects the symmetry axis corresponds to a solar eclipse and for the second type of orbits this point corresponds to a lunar eclipse. Mathematically the first type will correspond to the chords ending in the Lagrangian $L_S$ and the second to chords ending in $L_M$.
\end{remark}
\begin{proof}[Proof of Theorem~\ref{maxEx}]
The Idea to prove this statement is that we first define an easier but still homotopic Hamiltonian system and then use Theorem~\ref{tate} to calculate the $\mathbb{Z}_2$-equivariant Floer theory. The dimension of this homology will give us a lower bound for the number of orbits we are looking for. For this homotopic system we choose as Hamiltonian
\begin{align*}
\widetilde{H}: \mathbb{C}^2\to \mathbb{R},\ \ z\mapsto \left(\sum\limits_{i=1}^2\lvert z_i\rvert^2\right)-1
\end{align*}
on $(\mathbb{C}^2,\omega)$ with the standard symplectic form
\begin{align*}
\omega =\sum\limits_{i=1}^2\D x_i\wedge\D y_i, 
\end{align*}
note that on $\mathbb{C}^2$ we use the notation $z=x+iy$ where $x=\binom{x_1}{x_2}\in\mathbb{R}^2$ and $y=\binom{y_1}{y_2}\in\mathbb{R}^2$. Note that we can consider $(\mathbb{C}^2,\omega)$ to be the completion of a Liouville domain with the boundary being $\widetilde{H}^{-1}(0)$, since $\widetilde{H}^{-1}(0)$ is just $S^3$. But for a well-defined RF-homology we actually need a Hamiltonian function that is constant outside of a compact set. So take a smooth cut off function $\chi_{[0, 2]}\in C^\infty(\mathbb{R},\mathbb{R})$, which has support in $\left[0, 2\right]$ and is constant $1$ in $\left[0, 2-\epsilon\right]$. Then redefine the Hamiltonian as 
\begin{align*}
 \widetilde{H}: \mathbb{C}^2\to \mathbb{R},\ \ z\mapsto \left[\left(\sum\limits_{i=1}^2\rvert z_i\rvert^2\right)-1\right]\cdot \chi_{[0, 2]}(\|z\|)+ c_0\cdot\left(1-\chi_{[0, 2]}(\|z\|)\right),
\end{align*}
where $c_0\in\mathbb{R}$ is non zero. Note that around our energy hypersurface $\widetilde{H}^{-1}(0)$ the Hamiltonian function is still the same as before. The Hamiltonian vector field in this neighbourhood is 
then in real notation
\begin{align*}
X_H(x,y)= 2\binom{-y}{x}
\end{align*}
From this we easily see that on $\widetilde{H}^{-1}(0)=S^3$ we have
\begin{align*} 
\lambda(X_H)=\frac{1}{2}\sum\limits_{i=1}^{2}(2x_i^2+2y_i^2)=1,
\end{align*}
i.e. the Hamiltonian vector field coincides with the Reeb vector field. Now choose as Lagrangian subspace $L=\mathbb{R}^2\subset\mathbb{C}^2$, which fulfils $\lambda\big\vert_L=0$. So we have now all the necessary ingredients for the Lagrangian Rabinowitz Floer homology. Now we calculate the Conley-Zehnder index:\\
For a given $z\in S^3\cap L$ the corresponding chord is then simply
\begin{align*}
\phi(z,t)=
\begin{pmatrix}
	\cos(2t)\mathbb{1}_2 & \sin(2t)\mathbb{1}_2\\
	-\sin(2t)\mathbb{1}_2 & \cos(2t)\mathbb{1}_2
\end{pmatrix}
\begin{pmatrix}
	x\\y
\end{pmatrix}
\ \ \text{ for } t\in\left[0,\frac{m}{2}\pi\right],
\end{align*}
where for now $m\in\mathbb{N}$. To compute the Lagrangian CZ-index for those chords we use our Lagrangian $L$ as base point and the following path of Lagrangian subspaces:
\begin{align*}
\Gamma(t):= \Omega(t)L\ \ \text{ with }\ \ 
\Omega(t)=
\begin{pmatrix}
	\cos(2t)\mathbb{1}_2 & \sin(2t)\mathbb{1}_2\\
	-\sin(2t)\mathbb{1}_2 & \cos(2t)\mathbb{1}_2
\end{pmatrix}
\end{align*}
For the quadratic form take as Lagrangian subspaces $\mathbb{R}^2,\ i\mathbb{R}^2\subset\mathbb{C}^2$. To calculate this form we first need to find for every $\binom{x}{0}$ in $\mathbb{R}^2$ a $\binom{0}{v}$ in $i\mathbb{R}^2$ such that 
\begin{align*}
\binom{x}{0}+\binom{0}{v}=\binom{\cos(2t)y}{-\sin(2t) y},
\end{align*}
for an arbitrary $y\in\mathbb{R}^2$. This is fulfilled by $v=-\tan(2t)x$. Hence the quadratic form is around the intersection points $t=\frac{m}{2}\pi$ of $\Gamma(t)$ with $L$
\begin{align*}
Q_{\Gamma(t)}\left(\binom{x}{0}\right)&= \omega\left(\binom{x}{0},\binom{0}{-\tan(2t)x}\right)\\
&=\langle x,-\tan(2t)x\rangle.
\end{align*}
Hence we have for the crossing form
\begin{align*}
\frac{\D}{\D t}\bigg\vert_{t=\frac{m}{2}\pi} Q_{\Gamma(t)}\left(\binom{x}{0}\right)=\langle x,\frac{-1}{\cos^2(2t)}\mathbb{1}_2 x\rangle
\end{align*}
and its signum is $-2$. Therefore $\text{sign}C(\Gamma,L,t)$ for $t$ being a intersection point is $-2$ and for negative periods the minus sign becomes a plus sign.
Every chord is uniquely defined by its starting point on $S^3\cap L=S^{1}$ and its period $\tau$, hence these critical points form infinitely many copies of $S^{1}$ index by their shared period.  
So let now $x$ be a chord with period $\tau=\frac{m}{2}\pi$, then its Conley Zehnder index is given by
\begin{align*}
\mu_{CZ}(x)&=-\frac{2}{2}-2(m-1)-\frac{2}{2}\\
&= -2m.
\end{align*}

Now we want to introduce our symmetry, which is just simply multiplying by $-1$. Since the Hamiltonian only sees the norm of the points and the symplectic form is a linear two form, this is in fact a  symmetry of the Hamiltonian system and $J$ is clearly $\mathbb{Z}_2$-equivariant. Further we can interpret this as an action of the Lie group $\mathbb{Z}_2$, which is obviously free and proper. By the previous section this means that we have a well-defined $\mathbb{Z}_2$-equivariant RF-homology. 

To compute this homology we use Theorem~\ref{tate}: We already know that $\mathbb{Z}_2$ is a symmetry of the Hamiltonian system that acts free and that the Hamiltonian system itself is displaceable. Further, $\mathbb{R}^2\cap H^{-1}(0)=S^{1}$ is a one dimensional connected submanifold and the above considerations show now that the Conley-Zehnder index for non-constant chords fulfils the condition
\begin{align*}
\lvert\mu_{CZ}(x)\rvert\ge 2> \text{dim}\left(\mathbb{R}^2\cap H^{-1}(0)\right).
\end{align*}
Hence Theorem~\ref{tate} tells us that 
\begin{align*}
RFH^{\mathbb{Z}_2}_*(\mathbb{C}^2,\widetilde{H},L=\mathbb{R}^2)= TH_*(\mathbb{Z}_2,\mathbb{Z}_2). 
\end{align*}
Note that by \cite{brown1982a}[§VI.4] the Tate homology of $\mathbb{Z}_2$ (with $\mathbb{Z}_2$ coefficients) is just the group homology in positive degrees and the group cohomology in the negative degrees and since the classifying space of $\mathbb{Z}_2$ is $\mathbb{RP}^\infty$ the Tate homology is just $\mathbb{Z}_2$ for every index. The invariance property then tells us that 
\begin{align*}
RFH_*^{\mathbb{Z}_2}(\mathbb{C}^2,H,L_{\text{col}})\cong RFH^{\mathbb{Z}_2}_*(\mathbb{C}^2,\widetilde{H},\mathbb{R}^2),
\end{align*}
since the Hamiltonian of the restricted three body problem after Levi-Civita regularization (see equation~(\ref{3Ham})) is invariant under the $\mathbb{Z}_2$ action and so is obviously every Lagrangian vector space. We can further use the invariance to infer that
\begin{align*}
RFH_*^{\mathbb{Z}_2}(\mathbb{C}^2,H,L_{\text{col}})\cong RFH_*^{\mathbb{Z}_2}(\mathbb{C}^2,H,L_{\text{col}},\widehat{F}^t_S),
\end{align*}
because $\widehat{F}^t_S=\frac{\pi}{2}(z_2^2+w_2^2)\chi_{\widetilde{K}}(z^2,w^2)\zeta(t)$ is also $\mathbb{Z}_2$-invariant. It should be clear that $$\dim\left(RFH_*^{\mathbb{Z}_2}(\mathbb{C}^2,H,L_{\text{col}},\widehat{F}^t_S)\right)$$ still gives us a lower bound for the number of critical points of the perturbed Rabinowitz action functional $\mathcal{A}_{\widehat{F}^t_S}^H$ and by what we discussed at the beginning of section~\ref{ExI} all these critical points correspond to symmetric consecutive collision orbits which intersect the symmetry axis at a solar eclipse point. 

Note that the statement that the RFH of $\widetilde{H}$ and $H$ are isomorphic implicitly assumed that $\mathcal{A}_H$ is Morse-Bott, but this does not need to be the case. However, \cite[Theorem~B.1.]{cieliebak2009a} tells us that the Rabinowitz action functional is Morse-Bott for a generic choice of Hamiltonians and therefore we can find a sequence of Hamiltonian functions $H_n$ converging to $H$ such that they are Morse-Bott and fulfil 
\begin{align*}
RFH_*^{\mathbb{Z}_2}(\mathbb{C}^2,H_n,L_{\text{col}})\cong RFH^{\mathbb{Z}_2}_*(\mathbb{C}^2,\widetilde{H},\mathbb{R}^2).
\end{align*}
This again implies that for every Maslov index $\mu$ we find a trajectory $x^\mu_n$ of $H_n$ that starts in $L_{\text{col}}$ and ends in $L_S$. By using spectral numbers one can see that the action $\mathcal{A}_{H_n}\left(x^\mu_n\right)$ is bounded for all $n$ and since the period is equal to the action value the periods in the sequence $\left(x^\mu_n\right)_{n\in\mathbb{N}}$ are bounded. Hence, we can use Arzelà–Ascoli and see that the $x^\mu_n$ (up to taking a subsequence) converge to a trajectory $x_*$ of $H$. The Maslov index of $x_*$ now does not need to coincide with $\mu$, but it can at most change by plus or minus the half dimension of the Maslov pseudo cycle. So by taking the limit $H_n\to H$ the infinitely many chords of $H_n$ can not collapse into only finitely many of $H$. The same arguments of course also works for the other type of orbits using instead $RFH_*^{\mathbb{Z}_2}(\mathbb{C}^2,H,L_{\text{col}},\widehat{F}^t_M)$. This proves the theorem.
\end{proof}

\begin{remark}
It shouldn't be too surprising that using the symmetry of our Hamiltonian leads to a better understanding of the system, since this one of the key concepts in physics. In classical mechanics one usually uses symmetries to simplify the equations of motion, which makes it easier to compute the trajectories in the system. In Rabinowitz Floer homology the symmetry apparently has a slightly different effect. Incorporating the symmetry into our considerations doesn't make it easier to compute the homology, instead the predictive power of the homology increases drastically.
\end{remark}

\newpage
\bibliographystyle{alpha}
\bibliography{Kbib}{}

\begin{thebibliography}{AFKP12}

\bibitem[AF10]{albers2010leaf}
Peter Albers and Urs Frauenfelder.
\newblock Leaf-wise intersections and rabinowitz floer homology.
\newblock {\em Journal of Topology and Analysis}, 2(01):77--98, 2010.

\bibitem[AF12]{albers2012infinitely}
Peter Albers and Urs Frauenfelder.
\newblock Infinitely many leaf-wise intersections on cotangent bundles.
\newblock {\em Expositiones Mathematicae}, 30(2):168--181, 2012.

\bibitem[AFKP12]{albers2012}
Peter Albers, Urs Frauenfelder, Otto~Van Koert, and Gabriel~P Paternain.
\newblock Contact geometry of the restricted three-body problem.
\newblock {\em Communications on pure and applied mathematics}, 65(2):229--263,
  2012.

\bibitem[Bro82]{brown1982a}
Kenneth~S Brown.
\newblock {\em Cohomology of groups}, volume~87.
\newblock Springer Science \& Business Media, 1982.

\bibitem[CF09]{cieliebak2009a}
Kai Cieliebak and Urs Frauenfelder.
\newblock A floer homology for exact contact embeddings.
\newblock {\em Pacific journal of mathematics}, 239(2):251--316, 2009.

\bibitem[EGH00]{eliashberg2000introduction}
Yakov Eliashberg, A~Glvental, and Helmut Hofer.
\newblock Introduction to symplectic field theory.
\newblock In {\em Visions in mathematics}, pages 560--673. Springer, 2000.

\bibitem[FVK18]{frauenfelder2018a}
Urs Frauenfelder and Otto Van~Koert.
\newblock {\em The restricted three-body problem and holomorphic curves}.
\newblock Springer, 2018.

\bibitem[FZ19]{frauenfelder2019a}
Urs Frauenfelder and Lei Zhao.
\newblock Existence of either a periodic collisional orbit or infinitely many
  consecutive collision orbits in the planar circular restricted three-body
  problem.
\newblock {\em Mathematische Zeitschrift}, 291(1):215--225, 2019.

\bibitem[Mer14]{merry2014a}
Will~J Merry.
\newblock Lagrangian rabinowitz floer homology and twisted cotangent bundles.
\newblock {\em Geometriae dedicata}, 171(1):345--386, 2014.

\bibitem[Mos78]{moser1978fixed}
J{\"u}rgen Moser.
\newblock A fixed point theorem in symplectic geometry.
\newblock {\em Acta mathematica}, 141:17--34, 1978.

\bibitem[Tat52]{tate1952a}
John Tate.
\newblock The higher dimensional cohomology groups of class field theory.
\newblock {\em Annals of Mathematics}, pages 294--297, 1952.

\end{thebibliography}
\end{document}